\documentclass[a4paper, fleqn]{amsart}
\pdfoutput=1
\usepackage[utf8]{inputenc}
\usepackage[T1]{fontenc}
\usepackage[colorlinks]{hyperref}
\usepackage[capitalise]{cleveref}
\usepackage{microtype}
\usepackage{enumitem}

\hypersetup{
  linkcolor=[rgb]{0.3,0.3,0.6},
  citecolor=[rgb]{0.2, 0.6, 0.2},
  urlcolor=[rgb]{0.6, 0.2, 0.2}
}

\usepackage{mathtools, amsthm, amsfonts, amssymb}
\usepackage{commath}

\usepackage{accents}

\theoremstyle{plain}
\newtheorem{theorem}{Theorem}[section]
\newtheorem*{theorem*}{Theorem}

\newtheorem{lemma}[theorem]{Lemma}
\newtheorem{corollary}[theorem]{Corollary}

\theoremstyle{definition}
\newtheorem{definition}[theorem]{Definition}

\newtheorem*{problem*}{Problem}

\newtheorem{remark}[theorem]{Remark}

\DeclarePairedDelimiter\floor{\lfloor}{\rfloor}

\newcommand{\regularize}[1]{\underaccent{\wtilde}{#1}}

\DeclareMathOperator{\rank}{R}
\DeclareMathOperator{\subrank}{Q}

\DeclareMathAccent{\wtilde}{\mathord}{largesymbols}{"65}
\DeclareMathOperator{\asymprank}{\underaccent{\wtilde}{R}}
\DeclareMathOperator{\asympsubrank}{\underaccent{\wtilde}{Q}}

\newcommand{\FF}{\mathbb{F}}

\newcommand{\NN}{\mathbb{N}}

\newcommand{\ZZ}{\mathbb{Z}}
\newcommand{\RR}{\mathbb{R}}

\DeclareMathOperator{\Hom}{Hom}

\newcommand{\jeroen}[1]{#1}

\newcommand{\defin}[1]{#1}

\let\leqx\leqslant
\let\geqx\geqslant

\newcommand{\asympleqx}{%
  \mathrel{%
    \vphantom{\leqslant}%
    \smash{\vcenter{\doasympleqx}}%
  }%
}
\newcommand{\doasympleqx}{%
  \hbox{\ooalign{%
    \noalign{\kern.25ex}
    $\leqslant$\cr
    \noalign{\kern1.25ex}
    \smash{$\sim$}\cr
  }}%
}

\newcommand{\asympdomleq}{%
  \mathrel{%
    \vphantom{\preccurlyeq}%
    \smash{\vcenter{\doasympdomleq}}%
  }%
}
\newcommand{\doasympdomleq}{%
  \hbox{\ooalign{%
    \noalign{\kern.25ex}
    $\preccurlyeq$\cr
    \noalign{\kern1.25ex}
    \smash{$\sim$}\cr
  }}%
}

\newcommand{\asympasympdomleq}{%
  \mathrel{%
    \vphantom{\preccurlyeq}%
    \smash{\vcenter{\doasympasympdomleq}}%
  }%
}
\newcommand{\doasympasympdomleq}{%
  \hbox{\ooalign{%
    \noalign{\kern.25ex}
    $\preccurlyeq$\cr
    \noalign{\kern1.25ex}
    \smash{$\sim$}\cr
    \noalign{\kern0.5ex}
    \smash{$\sim$}\cr
  }}%
}

\newcommand{\eps}{\varepsilon}

\DeclareMathOperator{\characteristic}{char}

\newcommand{\domleq}{\preccurlyeq}

\newcommand{\semiring}[1]{#1}

\newcommand{\tand}{\textnormal{ and }}
\newcommand{\tor}{\textnormal{ or }}

\newcommand{\graphs}{\mathcal{G}}

\newcommand{\aspec}{\mathbf{X}}

\newcommand{\textgraphs}{\{\textnormal{graphs}\}}

\let\leqx\leqslant
\let\geqx\geqslant

\newcommand{\asympgeqx}{%
  \mathrel{%
    \vphantom{\geqslant}%
    \smash{\vcenter{\doasympgeqx}}%
  }%
}
\newcommand{\doasympgeqx}{%
  \hbox{\ooalign{%
    \noalign{\kern.25ex}
    $\geqslant$\cr
    \noalign{\kern1.25ex}
    \smash{$\sim$}\cr
  }}%
}

\newcommand{\asympasympleqx}{%
  \mathrel{%
    \vphantom{\preccurlyeq}%
    \smash{\vcenter{\doasympasympleqx}}%
  }%
}
\newcommand{\doasympasympleqx}{%
  \hbox{\ooalign{%
    \noalign{\kern.25ex}
    $\leqx$\cr
    \noalign{\kern1.25ex}
    \smash{$\sim$}\cr
    \noalign{\kern0.5ex}
    \smash{$\sim$}\cr
  }}%
}

\DeclareMathOperator{\matrixrank}{rank}

\title[The asymptotic spectrum of graphs and the Shannon capacity]{The asymptotic spectrum of graphs\\ and the Shannon capacity}
\author{Jeroen Zuiddam}
\address{Centrum Wiskunde \& Informatica, Science Park 123, Amsterdam, Netherlands}
\address{\textit{Current affiliation:} Institute for Advanced Study, 1 Einstein Drive, Princeton NJ, USA}
\email{jzuiddam@ias.edu}
\date{\today}

\begin{document}

\begin{abstract}
We introduce the asymptotic spectrum of graphs and apply the theory of asymptotic spectra of Strassen (J.~Reine Angew.\ Math.\ 1988) to obtain a new dual characterisation of the Shannon capacity of graphs. Elements in the asymptotic spectrum of graphs include the Lovász theta number, the fractional clique cover number, the complement of the fractional orthogonal rank and the fractional Haemers bound.
\end{abstract}

\maketitle

\section{Introduction}\label{intro}

\subsection{Shannon capacity of graphs}
This paper is about the Shannon capacity of graphs, which was introduced by Shannon in~\cite{MR0089131}. %
Let $G$ be a (finite simple) graph with vertex set $V(G)$ and edge set $E(G)$. An independent set \jeroen{(also called stable set)} in~$G$ is a subset of~$V(G)$ that contains no edges. The independence number or stability number $\alpha(G)$ is the cardinality of the largest independent set in $G$. For graphs $G$ and $H$, the and-product $G \boxtimes H$, also called the strong graph product, is defined by
\begin{align*}
V(G \boxtimes H) &= V(G) \times V(H)\\
E(G \boxtimes H) &= \bigl\{ \{(g,h), (g'\!, h') \} :  \bigl(\{g,g'\} \in E(G) \tand \{h,h'\} \in E(H)\bigr) \\
&\hspace{9em}\tor \bigl(\{g,g'\} \in E(G) \tand h=h'\bigr) \\ 
&\hspace{9em}\tor \bigl(g=g' \tand \{h,h'\} \in E(H) \bigr)\bigr\}.
\end{align*}
The Shannon capacity $\Theta(G)$ is defined as the limit 
\[
\Theta(G) \coloneqq \lim_{N \to \infty} \alpha(G^{\boxtimes N})^{1/N}.
\]
This limit exists and equals the supremum $\sup_N \alpha(G^{\boxtimes N})^{1/N}$ by Fekete's lemma: if~$x_1, x_2, x_3, \ldots \in \RR_{\geq 0}$ satisfy \jeroen{$x_{m+n} \geq x_m + x_n$, then $\lim_{n\to\infty} x_n/n = \sup_n x_n/n$.}

Computing the Shannon capacity is nontrivial already for small graphs. Lovász in \cite{lovasz1979shannon} computed the value $\Theta(C_5) = \sqrt{5}$, where $C_k$ denotes the $k$-cycle graph, by introducing and evaluating a new graph parameter $\vartheta$ which is now known as the Lovász theta number. For example the value of~$\Theta(C_7)$ is currently not known. The Shannon capacity $\Theta$ is not known to be hard to compute in the sense of computational complexity. On the other hand, deciding whether~\jeroen{$\alpha(G) \geq k$}, given a graph $G$ and~$k\in \NN$, is NP-complete~\cite{MR0378476}. %

\subsection{Result: new dual characterisation of Shannon capacity}
The result of this paper is a new dual characterisation of the Shannon capacity of graphs. This characterisation is obtained by applying Strassen's theory of asymptotic spectra~\cite{strassen1988asymptotic}, which in turn is based on the representation theorem of Kadison--Dubois~\cite{MR707730} (see also~\cite{MR1829790} and~\cite{MR2383959}). %

To state \jeroen{our result}
we need the standard notions graph homomorphism, graph complement and graph disjoint union.
Let $G$ and $H$ be graphs.
A \defin{graph homomorphism} $f : G \to H$ is a map $f : V(G) \to V(H)$ such that for all $u,v \in V(G)$, if~$\{u,v\} \in E(G)$, then $\{f(u), f(v)\} \in E(H)$. In other words, a graph homomorphism maps edges to edges.
The \defin{complement}~$\overline{G}$ of $G$ is defined by
\begin{gather*}
V(\overline{G}) = V(G)\\ 
E(\overline{G}) = \bigl\{\{u,v\} : \{u,v\} \not\in E(G), u\neq v\bigr\}.
\end{gather*}
We define the relation $\leqx$ on graphs as follows: let $G \leqx H$
if there is a graph homomorphism~$\overline{G} \to \overline{H}$ from the complement of $G$ to the complement of $H$.
The \defin{disjoint union}~$G \sqcup H$ is defined by
\begin{gather*}
V(G \sqcup H) = V(G) \sqcup V(H)\\ %
E(G \sqcup H) = E(G) \sqcup E(H). %
\end{gather*}
For $n \in \NN$ the \defin{complete graph} $K_n$ is the graph with $V(K_n) = [n] \coloneqq \{1,2,\ldots, n\}$ and $E(K_n) = \{\{i,j\} : i,j \in [n], i\neq j\}$. Thus $K_0 = \overline{K_0}$ is the empty graph and~$K_1 = \overline{K_1}$ is the graph consisting of a single vertex and no edges.

\jeroen{We define the relation $\asympleqx$ on graphs as follows: let $G \asympleqx H$ if there is a sequence~$(x_N) \in \NN^\NN$ with $x_N^{1/N}\to 1$ when $N \to \infty$ such that for every $N\in\NN$
\[
G^{\boxtimes N} \leqx (H^{\boxtimes N})^{\sqcup x_N} = \underbrace{H^{\boxtimes N} \sqcup \cdots \sqcup H^{\boxtimes N}}_{\smash{x_N}}
\]
holds. }

\begin{theorem}\label{basic_th}
Let $\semiring{S}\subseteq \textgraphs$ be a collection of graphs which is closed under the disjoint union~$\sqcup$ and the strong graph product~$\boxtimes$, and which contains the graph with a single vertex, $K_1$.
Define the asymptotic spectrum $\aspec(\semiring{S})$ as the set of all maps~$\phi : \semiring{S} \to \RR_{\geq 0}$ such that, for all~$G,H \in \semiring{S}$
\begin{enumerate}[label=\upshape(\arabic*)]
\item if $G\leqx H$, then $\phi(G) \leq \phi(H)$ \label{mainthmon}
\item $\phi(G \sqcup H) = \phi(G) + \phi(H)$\label{mainthadd}
\item $\phi(G \boxtimes H) = \phi(G)\phi(H)$\label{mainthmult}
\item $\phi(K_1) = 1$.\label{mainthnorm}
\end{enumerate}
Then we have
\begin{enumerate}[label=\upshape(\roman*)]
\item $G \asympleqx H$ iff\, $\forall \phi \in \aspec(\semiring{S})\,\, \phi(G) \leq \phi(H)$\label{asymp_i}
\item $\Theta(G) = \min_{\phi \in \aspec(\semiring{S})} \phi(G)$.\label{asymp_ii}
\end{enumerate}
\end{theorem}

\begin{remark}
Statement \ref{asymp_ii} of \cref{basic_th} is nontrivial in the sense that $\Theta$ is not an element of~$\aspec(\textgraphs)$. Namely,~$\Theta$ is not additive under $\sqcup$ by a result of Alon~\cite{alon1998shannon}, and $\Theta$ is not multiplicative under $\boxtimes$ by a result of Haemers~\cite{haemers1979some}. %
It turns out that the graph parameter~$G \mapsto \max_{\phi \in \aspec(\textgraphs)} \phi(G)$ is 
itself an element of~$\aspec(\semiring{\graphs})$, and is known as the fractional clique cover number $\overline{\chi}_f$ (see \cref{sec:frac} and e.g.~\cite[Eq.~(67.112)]{schrijver2003combinatorial}).
While writing this paper it was brought to the author's attention that in \jeroen{\cite[Example 8.25]{fritz2017resource}} a statement is proven that is slightly weaker than \cref{basic_th}, in the sense that condition \ref{mainthadd} is not imposed.
\end{remark}

\jeroen{
\begin{remark}
The relation $G \asympleqx H$ defined in \cref{basic_th} above can equivalently be characterised by $G^{\boxtimes N} \leqx H^{\boxtimes (N + o(N))}$. %
One may naturally consider the optimal coefficient $\alpha \in \RR$ for which an asymptotic inequality $G^{\boxtimes N} \leqx H^{\boxtimes (\alpha N + o(N))}$ holds, as is also studied in~\cite{wang2017graph}. Statement \ref{asymp_i} of \cref{basic_th} naturally implies a characterisation of this optimal rate in terms of the asymptotic spectrum.
\end{remark}
}

\subsection{Known elements in the asymptotic spectrum of graphs}
Several graph parameters from the literature can be shown to be in the asymptotic spectrum of graphs $\aspec(\textgraphs)$.
The elements in~$\aspec(\{\textnormal{graphs}\})$ that the author is currently aware of are: the Lovász theta number $\vartheta$ \cite{lovasz1979shannon}, the fractional clique cover number~$\overline{\chi}_{\smash{f}}$, the complement of the fractional orthogonal rank, $\overline{\xi_f}$%
~\cite{cubitt2014bounds}, and for each field~$\FF$ the fractional Haemers bound $\rank^\FF_f$ \cite{blasiak2013graph,bukh2018fractional}. Bukh and Cox in~\cite{bukh2018fractional} prove a separation result which implies that the fractional Haemers bounds provide an infinite family in~$\aspec(\{\textnormal{graphs}\})$!

\subsection{Strassen: asymptotic rank and asymptotic subrank of tensors}\label{tensors}
Volker Strassen developed the theory of asymptotic spectra in the context of tensors in the series of papers~\cite{Strassen:1986:AST:1382439.1382931, strassen1987relative, strassen1988asymptotic, strassen1991degeneration}. For any family~$\semiring{S}$ of tensors  (of some fixed order, over some fixed field, but with arbitrary dimensions) which is closed under tensor product and direct sum and which contains the ``diagonal tensors'', there is an asymptotic spectrum $\aspec(\semiring{S})$ characterising the ``asymptotic restriction preorder'' on~$\semiring{S}$. The asymptotic restriction preorder is closely related to the asymptotic rank and asymptotic subrank of tensors. Understanding these notions is in turn the key to understanding the computational complexity of matrix multiplication (see also~\cite{burgisser1997algebraic}), Strassen's original motivation.

Strassen constructed a collection of elements in $\aspec(\semiring{S})$ for the family $\semiring{S}$ of ``oblique tensors'', a strict subset of all tensors \cite{strassen1991degeneration}. These elements are called the support functionals.
Interestingly, the recent breakthrough result on the cap set problem~\cite{MR3583358,tao} can be proven using these support functionals~\cite{christandl2017universalproc}.
Only recently, Christandl, Vrana and Zuiddam in \cite{christandl2017universalproc} constructed an infinite family of elements in $\aspec(\{\textnormal{complex tensors of order $k$}\})$, called the quantum functionals.

\section{Asymptotic spectra}
We discuss the theory of asymptotic spectra, following \cite{strassen1988asymptotic}, \jeroen{but in the language of semirings instead of rings.}
Let~$(S, +, \cdot, 0, 1)$ be a commutative semiring, meaning that~$S$ is a set with a binary addition operation $+$, a binary multiplication operation~$\cdot$, and elements $0,1\in S$, such that for all $a,b,c \in S$
\begin{enumerate}[label=\upshape(\arabic*)]
\item $+$ is associative: $(a+b)+c = a + (b+c)$
\item $+$ is commutative: $a+b = b+a$
\item $0+a = a$
\item $\cdot$ is associative: $(a\cdot b)\cdot c = a \cdot (b \cdot c)$
\item $\cdot$ is commutative: $a\cdot b = b \cdot a$
\item $1\cdot a = a$
\item $\cdot$ distributes over $+$: $a\cdot (b+c) = (a\cdot b) + (a \cdot c)$
\item $0 \cdot a = 0$.
\end{enumerate}
For $n \in \NN$ we denote the sum of $n$ ones $1 + \cdots + 1 \in S$ by $n$.

Let $\leqx$ be a preorder on $S$, i.e.~$\leqx$ is a relation on $S$ such that for all $a,b,c \in S$
\begin{enumerate}[label=\upshape(\arabic*)]
\item $\leqx$ is reflexive: $a\leqx a$
\item $\leqx$ is transitive: $a\leqx b$ and $b\leqx c$ implies $a\leqx c$.
\end{enumerate}
\begin{definition}\label{strassendef}
A preorder $\leqx$ on $S$ is a \emph{Strassen preorder} if
\begin{enumerate}[label=\upshape(\arabic*)]
\item $\forall n, m \in \NN$\, $n\leq m$ in $\NN$ iff $n \leqx m$ in $S$ \label{strcond1}
\item $\forall a,b,c,d \in S$\, if $a\leqx b$ and $c\leqx d$, then $a + c \leqx b + d$ and $ac \leqx bd$\label{strcond2}
\item $\forall a,b \in S, b\neq 0$\, $\exists r \in \NN$\, $a \leqx rb$.\label{strcond3}
\end{enumerate}
\end{definition}

Let $S$ be a commutative semiring %
and let $\leqx$ be a Strassen preorder on $S$. 
We will use $\leq$ to denote the usual preorder on $\RR$.
Let~$\RR_{\geq0}$ be the semiring of non-negative real numbers. 
Let $\aspec(S, \leqx)$ be the set of $\leqx$-monotone semiring homomorphisms from~$S$ to $\RR_{\geq0}$,
\[
\aspec(S) \coloneqq \aspec(S, \leqx) \coloneqq \{ \phi \in \Hom(S, \RR_{\geq 0}) : \forall a,b \in S\,\, a\leqx b \Rightarrow \phi(a) \leq \phi(b) \}.
\]
We call $\aspec(S, \leqx)$ the \emph{asymptotic spectrum} of $(S, \leqx)$.
\jeroen{Note that for every~$\phi \in \aspec(S, \leqx)$ holds $\phi(1) = 1$ and thus $\phi(n) = n$ for all $n \in \NN$.}
For $a,b \in S$, let $a \asympleqx b$ if there is a sequence $(x_N) \in \NN^\NN$ with $x_N^{\smash{1/N}}\to 1$ when $N \to \infty$ such that for all~$N\in\NN$ we have $a^{N} \leqx b^{N}  x_N$. Fekete's lemma implies that in the definition of~$\asympleqx$ we may equivalently replace the requirement~$x_N^{\smash{1/N}}\to 1$ when $N\to\infty$ by~$\inf_N x_N^{\smash{1/N}} = 1$.
We call $\asympleqx$ the \emph{asymptotic preorder} induced by~$\leqx$.

\jeroen{
In terms of rings, the main theorem in the theory of asymptotic spectra is the following.

\begin{theorem}[{\cite[Cor.~2.6]{strassen1988asymptotic}}]\label{src}
Let $(R, +, \,\cdot\,, 0, 1)$ be a commutative ring and let~$R_+\subseteq R$ be a subset such that $0,1 \in R_+$, %
$-1\not\in R_+$ and $R = R_+ - R_+$, and such that~$(R_+, +,\,\cdot\,, 0,1)$ forms a semiring. %
Let~${\leqx}$ be a Strassen preorder on $R_+$ and let~${\asympleqx}$ be the asymptotic preorder induced by ${\leqx}$.
Let $X$ be the set of ring homomorphisms~$R \to \RR$ that are $\leqx$-monotone on $R_+$, i.e.\
\[
X = \{ \phi \in \Hom(R, \RR) : \forall a, b \in R_+\,\, a \leqx b \Rightarrow \phi(a) \leq \phi(b)\}.
\]
Then for $a,b \in R_+$ holds $a \asympleqx b \textnormal{ iff } \forall \phi \in X\,\, \phi(a) \leq \phi(b)$.
\end{theorem}

Let $R$ be the Grothendieck ring of $S$. (As group under addition, $R$ is the quotient of the free abelian group on symbols $[a]$, $a \in S$, and the subgroup generated by the elements $[a+b] - [a] - [b]$, $a,b \in S$. To make $R$ a ring, multiplication is defined by setting $[a][b] = [ab]$, $a,b \in S$ and extending $\ZZ$-linearly.)
To study the asymptotic preorder~$\asympleqx$ on the semiring $S$, the natural approach is to apply \cref{src} to $R$. The canonical semiring homomorphism $S \to R : a \mapsto [a]$ is, however, not injective in general, which a priori seems an issue. Namely, $[a] = [b]$ if and only if there exists an element $c \in S$ such that $a+c = b+c$.

To see that noninjectivity is not an issue we use the following lemma. %
Proving the lemma is routine if done in the suggested order. A proof can be found in %
\cite[Chapter~2]{phd}.
\begin{lemma}\label{biglem}
Let $\domleq$ be a Strassen preorder on a commutative semiring $T$.
Let $\asympdomleq$ be the asymptotic preorder induced by $\domleq$ and let $\asympasympdomleq$ be the asymptotic preorder induced by~$\asympdomleq$. 
 Then the following are true.
\begin{enumerate}[label=\upshape(\roman*)]
\item Also $\asympdomleq$ is a Strassen preorder on $T$. \label{23}
\item For any $a_1, a_2 \in T$, if $a_1 \asympasympdomleq a_2$, then $a_1 \asympdomleq a_2$.\label{asympasymp}
\item For any $a_1, a_2, b \in T$ we have $a_1 + b \asympdomleq a_2 + b$ iff $a_1\asympdomleq a_2$. \label{add_reg}
\end{enumerate}
\end{lemma}

Let $R_+ \subseteq R$ be the image of $S$ under the canonical map $S \to R$. Define the relation $\asympleqx$ on $R_+$ by letting $[a] \asympleqx [b]$ if $a \asympleqx b$ for any $a,b \in S$. One verifies using \cref{biglem} that~$\asympleqx$ is a Strassen preorder on $R_+$ and that for any $a,b \in S$ holds~$[a] \asympleqx [b]$ if and only if $a \asympleqx b$.  Moreover, one verifies that $\asympasympleqx$ coincides with $\asympleqx$ on $R_+$. Applying \cref{src} to $R_+ \subseteq R$ with the Strassen preorder $\asympleqx$, and using the fact that $\aspec(S, \leqx) = \aspec(S, \asympleqx)$, yields the following corollary.

\begin{corollary}%
\label{str_mainth}
Let $S$ be a commutative semiring and let $\leqx$ be a Strassen preorder on $S$. Then
\[
\forall a,b \in S\quad
a\asympleqx b \,\,\textnormal{ iff }\,\, \forall\phi \in \aspec(S, \leqx)\,\, \phi(a) \leq \phi(b).
\]
\end{corollary}
}

\jeroen{
\begin{remark}
One can prove that the semiring of graphs $S$ that we will consider in \cref{sec:graphs} is, in fact, additively cancellative, which means that for any $a,b,c \in S$ holds $a+c = b+c$ if and only if $a = b$. 
In that case the canonical map $S \to R$ is injective and the statement of \cref{str_mainth} follows directly from \cref{src}. The above argument shows that $S$ being additively cancellative is not necessary for the conclusion of \cref{str_mainth} to hold.
\end{remark}}

\begin{remark}
Alternatively, to prove \cref{str_mainth} one may integrate the proofs of Strassen \cite{strassen1988asymptotic} and Becker--Schwartz \cite{MR707730} in a manner that avoids passing to the Grothendieck ring altogether. Such a proof can be found in the author's PhD thesis~\cite[Chapter 2]{phd}.
\end{remark}

\begin{remark}

For $a \in S$, let 
$\hat{a} : \aspec(S) \to \RR_{\geq 0} : \phi \mapsto \phi(a)$.
Let $\RR_{\geq 0}$ have the Euclidean topology.
Endow $\aspec(S)$ with the weak topology with respect to the maps~$\hat{a}$ for~$a \in S$. %
It can be shown that with this topology $\aspec(S)$ is a nonempty compact Hausdorff space, see e.g.~\cite{strassen1988asymptotic}.
\end{remark}

\section{Rank and subrank}\label{str_sec:rank}
Strassen in \cite{strassen1988asymptotic} studied the asymptotic rank and asymptotic subrank of tensors (cf.~\cref{tensors}).
We generalise the notions of asymptotic rank and asymptotic subrank of tensors to arbitrary semirings $S$ with a Strassen preorder $\leqx$.
Let $a \in S$. %
Define the \defin{rank} 
\[
\rank(a) \coloneqq \min \{ r \in \NN : a \leqx r\}
\]
and the \defin{subrank} 
\[
\subrank(a) \coloneqq \max\{ s \in \NN : s \leqx a\}.
\]
Then $\subrank(a) \leq \rank(a)$. %
Define the \defin{asymptotic rank}
\begin{equation*}\label{str_asympr}
\asymprank(a) \coloneqq \lim_{N \to \infty} \rank(a^N)^{1/N}.
\end{equation*}
When $a = 0$ or $a\geq1$ define the \defin{asymptotic subrank}
\begin{equation*}\label{str_asymps}
\asympsubrank(a) \coloneqq \lim_{N \to \infty} \subrank(a^N)^{1/N}.
\end{equation*}
By Fekete's lemma these limits exist, and
asymptotic rank is an infimum and asymptotic subrank is a supremum as follows,
\begin{align*}
\asymprank(a) &= \inf_N \rank(a^N)^{1/N}\\
\asympsubrank(a) &= \sup_N \subrank(a^N)^{1/N} \textnormal{ when $a = 0$ or $a\geq1$}.
\end{align*}
\cref{str_mainth} implies that the asymptotic rank and asymptotic subrank have the following dual characterisation in terms of the asymptotic spectrum. %
The proof is essentially the same as the proof of \cite[Th.~3.8]{strassen1988asymptotic}. %

\begin{corollary}[cf.~{\cite[Th.~3.8]{strassen1988asymptotic}}]\label{Qmin}\label{str_rankthm}
\jeroen{Let $a \in S$ such that %
$a\geqx 1$ and such that there exists an element $k\in \NN$ with $a^k \geqx 2$.} Then
\begin{align*}
\asympsubrank(a) &= \min_{\phi \in \aspec(S)} \phi(a).
\end{align*}
\jeroen{Let $a \in S$ such that $a \geqx 1$. Then}
\begin{align*}
\asymprank(a) &= \max_{\phi \in \aspec(S)} \phi(a).
\end{align*}
\end{corollary}

\begin{proof}
Let $\phi \in \aspec(S)$. For $N \in \NN$, $\subrank(a^N) \leq \jeroen{\phi(a^N) =  \phi(a)^N}$. Therefore $\asympsubrank(a) \leq \phi(a)$. We conclude~$\asympsubrank(a) \leq \min_{\phi \in \aspec(S)} \phi(a)$.
It remains to prove $\asympsubrank(a) \geq \min_{\phi \in \aspec(S)} \phi(a)$.
Let~$a\geqx 1$ and~$a^k \geqx 2$.
Let $y \coloneqq \min_{\phi \in \aspec(S)} \phi(a)$. From $a\geqx 1$ follows \jeroen{that for all~$\phi \in \aspec(S)$ holds $\phi(a) \geq \phi(1) = 1$}, and so~$y \geq1$. By definition of $y$ we have
\[
\forall \phi \in \aspec(S)\quad \phi(a) \geq y.
\]
Take the $m$th power on both sides,
\[
\forall \phi \in \aspec(S), m \in \NN\quad \phi(a^m) \geq y^m.
\]
Take the floor on the right-hand side,
\[
\forall \phi \in \aspec(S), m \in \NN\quad \phi(a^m) \geq \floor{y^m}. %
\]
Apply \cref{str_mainth} to pass to the asymptotic preorder
\[
\forall m \in \NN\quad a^m \asympgeqx \floor{y^m}.
\]
Then, by the definition of asymptotic preorder,
\[
\forall m,N \in \NN\quad a^{m N} 2^{\eps_{m,N}} \geqx \floor{y^m}^{N}\quad\textnormal{for some}\quad \eps_{m,N} \in o(N).
\]
Now we use $a^k \geqx 2$ to get
\[
\forall m,N \in \NN\quad a^{m N + k \eps_{m,N}} \geqx \floor{y^m}^{N}.
\]
Then
\[
\forall m,N \in \NN\quad \subrank(a^{mN + k \eps_{m,N}})^{\smash{\tfrac{1}{mN + k\eps_{m,N}}}} \geq \floor{y^m}^{\smash{\tfrac{N}{mN + k \eps_{m,N}}}}.
\]
Choose $m = m(N)$ with $m(N)\to\infty$ as $N \to \infty$ and $\eps_{m(N),N} \in o(N)$ to obtain~$\asympsubrank(a) = \sup_N \subrank(a^N)^{1/N} \geq y$.
This proves the first statement.

The second statement is proven similarly.
\end{proof}

\section{The asymptotic spectrum of graphs}\label{sec:graphs}
We apply the theory of the previous two sections to graphs.
Let~$\graphs$ be the set of isomorphism classes of finite simple graphs. Let~$\boxtimes$ be the strong graph product, let~$\sqcup$ be the disjoint union of graphs, and
let $K_n$ be the complete graph with $n$ vertices, as defined in the introduction.

\begin{lemma}\label{lem1}
The set $\graphs$ with addition $\sqcup$, multiplication $\boxtimes$, additive unit $K_0$ and multiplicative unit $K_1$ is a commutative semiring. %
\end{lemma}

The complements $\overline{K_0}, \overline{K_1}, \overline{K_2}, \ldots$ of the complete graphs behave like $\NN$ in $\graphs$.

Let $G, H \in \graphs$. Let $G \leqx H$ if there is a graph homomorphism $\overline{G} \to \overline{H}$. In other words, $G \leqx H$ iff there is a map $\phi : V(G) \to V(H)$ such that $\forall u, v \in V(G)$ with~$u\neq v$, if $\{u,v\} \not\in E(G)$, then $\{\phi(u), \phi(v)\} \not\in E(H)$.

\jeroen{The reader readily verifies the following lemma.}
\begin{lemma}\label{lem2}
The relation $\leqx$ is a Strassen preorder on $\graphs$. That is: %
\begin{enumerate}[label=\upshape(\roman*)]
\item For $n, m \in \NN$, $\overline{K_n} \leqx \overline{K_m}$ iff $n \leq m$.\label{grpr_i}
\item If $A \leqx B$ and $C \leqx D$, then $A\sqcup C \leqx B\sqcup D$ and $A\boxtimes C \leqx B\boxtimes D$.\label{welldef}\label{grpr_ii} %
\item For $A,B \in \graphs$, if $B \neq K_0$, then there is an $r \in \NN$ with $A \leqx \overline{K_r} \boxtimes B$.\label{grpr_iii}
\end{enumerate}
\end{lemma}

Let~$\asympleqx$ be the asymptotic preorder induced by~$\leqx$.
As in the previous section, let~$\subrank(G) = \max \{m \in \NN : m \leqx G\}$ be the subrank. One verifies that~$\subrank(G)$ equals the independence number~$\alpha(G)$. %
Let~$\rank(G) = \min\{ n \in \NN : G \leqx n\}$ be the rank. One verifies that $\rank(G)$ equals the clique cover number $\overline{\chi}(G)$, i.e.~the chromatic number of the complement~$\chi(\overline{G})$.

Recall the definition of the Shannon capacity $\Theta(G) \coloneqq \lim_{N \to \infty} \alpha(G^{\boxtimes N})^{1/N}$. Thus~$\Theta(G)$ equals the asymptotic subrank $\asympsubrank(G)$.
One analogously defines the asymptotic clique cover number $\regularize{\overline{\chi}}(G) = \lim_{N \to \infty}\overline{\chi}(G^{\boxtimes N})^{1/N}$, which equals the asymptotic rank $\asymprank(G)$.
It is a nontrivial fact that the parameter $\regularize{\overline{\chi}}(G)$ equals the so-called fractional clique cover number~$\overline{\chi}_f(G)$, see~\cref{sec:frac}. %

\begin{proof}[\upshape\bfseries Proof of \cref{basic_th}]
Let $S\subseteq \graphs$ be a semiring.
By \cref{lem1} and \cref{lem2} we may apply \cref{str_mainth}. This gives statement \ref{asymp_i} of \cref{basic_th}.
Let~$G\in S$.
If $G = K_0$, then $\phi(G) = 0 = \asympsubrank(G)$ for any $\phi \in \aspec(S)$.
If $K_1 \leqx G \leqx K_1$, then~$\phi(G) = 1 = \asympsubrank(G)$ for any $\phi \in \aspec(S)$.
Otherwise %
$G \geqx \overline{K_2}$. Then we may apply \cref{Qmin}. This gives statement \ref{asymp_ii} of \cref{basic_th}.
\end{proof}

We are left with a clear goal: explicitly describe the asymptotic spectrum of graphs~$\aspec(\graphs)$.

\section{Known elements in the asymptotic spectrum of graphs}

We finish with an overview of some known elements in the asymptotic spectrum of graphs~$\aspec(\graphs)$. 
\subsection{Lovász theta number}
\jeroen{For any real symmetric matrix~$A$ let $\Lambda(A)$ be the largest eigenvalue. 
The Lovász theta number $\vartheta(G)$ is defined as
\[
\vartheta(G) \coloneqq \min\{ \Lambda(A) : A \in \RR^{V(G)\times V(G)} \textnormal{ symm.}, \{u,v\}\not\in E(G)\Rightarrow A_{uv} = 1\}.
\]}%
The parameter $\vartheta(G)$ was introduced by Lovász in \cite{lovasz1979shannon}. We refer to~\cite{knuth1994sandwich} and~\cite{schrijver2003combinatorial} for a survey.
It follows from now well-known properties that
$\vartheta \in \aspec(\graphs)$. %

\subsection{Fractional graph parameters}\label{sec:frac}
Besides the Lovász theta number there are several elements in $\aspec(\textgraphs)$ %
that are naturally obtained as fractional versions of sub-multiplicative, sub-additive, $\leqx$-monotone maps $\graphs \to \RR_{\geq 0}$.
For any map $\phi : \graphs \to \RR_{\geq0}$ we define a fractional version $\phi_f$ by
\[
\phi_f(G) = \inf_d \frac{\phi\bigl( \overline{\overline{G} \boxtimes K_d} \bigr)}{d}.
\]
We will discuss several fractional parameters from the literature. %
\subsubsection{Fractional clique cover number}

We consider the fractional version of the clique cover number~$\overline{\chi}(G) = \chi(\overline{G})$. %
It is well-known that
$\overline{\chi}_f \in \aspec(\graphs)$, see e.g.~\cite{schrijver2003combinatorial}.
The fractional clique cover number $\overline{\chi}_f$ in fact equals the asymptotic clique cover number~$\regularize{\overline{\chi}}(G) = \lim_{N \to \infty}\overline{\chi}(G^{\boxtimes N})^{1/N}$ which we introduced in the previous section, see~\cite{mceliece1971hide} and also~\cite[Th.~67.17]{schrijver2003combinatorial}.

\subsubsection{Fractional Haemers bound}

Let $\matrixrank(A)$ denote the matrix rank of any matrix~$A$.
For any set $C$ of matrices %
define $\matrixrank(C) \coloneqq \min\{\matrixrank(A): A \in C\}$. 
\jeroen{For a field $\FF$ and a graph $G$ define the set of matrices
\begin{align*}
M^\FF(G) \coloneqq \{ A \in \FF^{V(G) \times V(G)} : \forall u,\!v\,\, A_{vv} \neq 0, \{u,v\} \not\in E(G) \Rightarrow A_{uv} = 0\}.
\end{align*}}%
Let
$\rank^\FF(G) \coloneqq \matrixrank(M^\FF(G))$. %
The parameter $\rank^{\smash{\FF}}(G)$ was introduced by Haemers in~\cite{haemers1979some} and is known as the Haemers bound.
The fractional Haemers bound~$\rank^{\smash{\FF}}_{\smash{f}}$ was studied by Anna Blasiak in~\cite{blasiak2013graph} and was recently shown to be~$\boxtimes$-mul\-ti\-pli\-ca\-tive by Bukh and Cox in~\cite{bukh2018fractional}. From this it is not hard to prove that~$\rank^{\smash{\FF}}_{\smash{f}} \in \aspec(\graphs)$. %
Bukh and Cox in \cite{bukh2018fractional} furthermore prove a separation result: for any field $\FF$ of nonzero characteristic and any $\eps>0$, there is a graph~$G$ such that for any field~$\FF'$ with $\characteristic(\FF)\neq\characteristic(\FF')$ the inequality~$\rank^\FF_f(G) <\eps \rank^{\FF'}_f(G)$ holds. This separation result implies that there are infinitely many elements in~$\aspec(\graphs)$!

\subsubsection{Fractional orthogonal rank}

In \cite{cubitt2014bounds} the orthogonal rank $\xi(G)$ and its fractional version the projective rank $\xi_f(G)$ are studied. 
It easily follows from results in~\cite{cubitt2014bounds} that %
$G\mapsto \xi_f(\overline{G})$ is in $\aspec(\graphs)$.

\subsubsection*{\upshape\bfseries Acknowledgements}
The author thanks Harry Buhrman, Matthias Christandl, Péter Vrana, Jop Briët, Dion Gijswijt, Farrokh Labib, Māris Ozols, Michael Walter, Bart Sevenster, Monique Laurent, Lex Schrijver, Bart Litjens and the members of the A\&C PhD\,\&\,postdoc seminar at CWI for useful discussions and encouragement. The author is supported by~NWO (617.023.116) and the QuSoft Research Center for Quantum Software. The author initiated this work when visiting the Centre for the Mathematics of Quantum Theory (QMATH) at the University of Copenhagen.

\raggedright
\bibliographystyle{alphaurlpp}
\bibliography{all}%

\end{document}